\documentclass[12pt,a4paper]{amsart}
\usepackage[latin1]{inputenc}
\usepackage[english]{babel}
\usepackage{amsmath,amstext,amsthm,amsfonts,amssymb,amscd}
\usepackage[dvips]{graphicx}
\usepackage{color}
\usepackage{hyperref}

\usepackage{cleveref}

\theoremstyle{plain}
\newtheorem{theorem}{Theorem}[section]
\newtheorem{proposition}[theorem]{Proposition}
\newtheorem{definition}{Definition}
\newtheorem{remark}{Remark}

\newtheorem{lemma}[theorem]{Lemma}
\newtheorem{corollary}[theorem]{Corollary}
\newtheorem{theorema}{Theorem}

\begin{document}
	
	\title[Subcohomology and a Livsic Theorem for Zooming Systems]{Subcohomology and a Livsic Theorem \\ for Zooming Systems}

	\author[L. Mbarki]{Lamine Mbarki}
	\address{Lamine Mbarki, Mathematics Department, Faculty of Sciences Tunis, University of Tunis El Manar, Tunis, Tunisia} \email{mbarki.lamine2016@gmail.com}
	
	\author[E. Santana]{Eduardo Santana}
	\address{Eduardo Santana, Universidade Federal de Alagoas, 57200-000 Penedo, Brazil}
	\email{jemsmath@gmail.com}

	\thanks{The second author was partially supported by CNPq,   under the project with reference 409198/2021-8.} 
    \thanks{2020 Mathematics Subject Classification Primary: 37A05 Secondary: 37B99.}

	\date{\today}

	\maketitle
	
	\begin{abstract} 
		In the context of continuous zooming systems $f:M \to M$ on a compact metric space $M$, which include the non-uniformly expanding ones, possibly with the presence of a critical set, with the zooming set dense in $M$, we prove that any Hölder potential $\phi : M \to \mathbb{R}$ for which the integrals $\int \phi d\mu \geq 0$ with respect to any $f$-invariant probability $\mu$, admits a continuous function $\lambda_{0} : M \to \mathbb{R}$ (which can be Hölder if some integral is positive) such that
		\[
		\phi \geq \lambda_{0}- \lambda_{0} \circ f.
		\]
		This extends a result in \cite{BJ} for $C^{1}$-expanding maps on the circle $\mathbb{T} = \mathbb{R}/\mathbb{Z}$ to important classes of maps as uniformly expanding, local diffeomorphisms with non-uniform expansion, Viana maps, Benedicks-Carleson maps and Rovella maps. We also give an example beyond the exponential contractions context.
		
		Moreover, in the case of the integrals $\int \phi d\mu = 0$ with respect to any $f$-invariant probability $\mu$ and the set of periodic points to be dense in $M$, we obtain a version of the Livsic Theorem, that is, the functions $\lambda_{0}$ can be taken such that
		\[
		\phi = \lambda_{0}- \lambda_{0} \circ f.
		\]
		Additionally, we also prove that the measure which maximizes the integrals is unique for a residual set of potentials.
	\end{abstract}

	\bigskip



	\section{Introduction}
	
	Given a compact metric space $M$ and a continuous map $f:M \to M$, it is a trivial task to verify that for any $f$-invariant probability $\mu$ we have the integral $\int (\alpha - \alpha \circ f) d\mu = 0$ for any continuous function $\alpha:M \to \mathbb{R}$. Also, if $x \in M$ is a periodic point with period $n \in \mathbb{N}$, then the sum $\sum_{i=0}^{n-1} (\alpha - \alpha \circ f)(f^{i}(x))=0$. It is readily obtained from the previous remark, by considering the following $f$-invariant probability:
	\[
	\mu = \frac{1}{n}\sum_{i=0}^{n-1}\delta_{f^{i}(x)}.
	\]
	
	\bigskip
	
	The question about the converse of this fact leads to the well known Livsic Theorem (see \cite{HK}, Theorem 19.2.1), which states:
	
	\begin{theorem}[Livsic Theorem]\label{Livsic1}
		Let $M$ be a Riemannian manifold, $U \subset M$ open, $f:U \to M$ a smooth embedding, $\Lambda \subset U$ a compact topologically transitive hyperbolic set, and $\varphi: \Lambda \to \mathbb{R}$ Hölder continuous. Suppose that for every $x \in \Lambda$ such that $f^{n}(x) = x$ we have $\sum_{i=0}^{n-1}\varphi(f^{i}(x)) = 0$. Then there exists a continuous function $\Phi:\Lambda \to \mathbb{R}$ such that $\varphi = \Phi - \Phi \circ f$. Moreover, $\Phi$ is unique up to an additive constant and Hölder with the same exponent as $\varphi$.
	\end{theorem}
		
		There are variations of this classical result, as can be seen, for example, in the Introduction of \cite{BJ}. Livsic Theorem is stated as follows:

	    \begin{theorem}[Livsic Theorem]\label{Livsic2}
	    	Let $\mathbb{T} = \mathbb{R}/\mathbb{Z}$ and $T: \mathbb{T} \to \mathbb{T}$ be a $C^{\omega}$ expanding map. Let $f: \mathbb{T} \to \mathbb{R}$ be $C^{k}$ for some $k = 1,2,\dots, \infty$ (resp. $\beta$-Hölder for some $0< \beta \leq 1$), and suppose that $\int f d\mu = 0$ for all $T$-invariant probability $\mu$. Then there exists a $C^{k}$ (resp. $\beta$-Hölder) continuous function $\varphi$ such that $f = \varphi - \varphi \circ T$.
	    \end{theorem}
	
	Also in \cite{BJ}, Theorem A, the following related result appears:
	
	\begin{theorem}
		Let $T:\mathbb{T} \to \mathbb{T}$ be a $C^{1}$ expanding map. Let $f : \mathbb{T} \to \mathbb{R}$ be a continuous $\beta$-Hölder function for some $0 <\beta \leq 1$, and suppose that $\int f d\mu \geq 0$ for all $T$-invariant probability $\mu$. Then there exists a $\beta$-Hölder function $\varphi : \mathbb{T} \to \mathbb{R}$ such that $f \geq \varphi - \varphi \circ T$.
	\end{theorem}
	
	    This theorem has been stated and proved independently by several authors. It first appears in an unpublished manuscript by Conze-Guivarc'h in \cite{CG}, where it is proved using thermodynamic formalism; the same approach is used by Savchenko in \cite{S}. More direct proofs, which do no use the Ruelle transfer operator, can be found in \cite{Bo1}, \cite{Bo2} and \cite{CLT}.
	
	    In this work, we extend this result to the context of zooming systems, which includes the non-uniformly expanding maps, by using techniques that can be seen in \cite{CLT}, Proposition 11. Our approach is constructing a continuous potential on the dense zooming set and we are able to extend such a potential to the whole space. After that, by denseness of the Hölder functions, we obtain the result. In the case of all integrals to be zero, we obtain a version of the Livsic Theorem. We also prove that the maximizing measure is unique for a residual set of continuous potentials.
	
	    Our work is organised as follows. In section 2, we begin by giving some preliminaries and notations that will be useful for the remainder of the work and stating our main result. In section 3, we proceed with the proof of our first main result and the proof of our second main result in section 4. In section 5, we finish our paper by giving some applications. We stress that our result extends to several important classes of examples beyond expanding maps on the circle.
	
	\bigskip
	
	
	From now and on, we proceed with definitions and statements. We begin by defining zooming systems as can be seen in \cite{Pi1}. It also can be seen in \cite{Sa}.
	
	\vspace{1 cm}
	
	\section{Preliminaries and Main Results}

	\subsection{Zooming Systems}The zooming times generalize hyperbolic times beyond the exponential context. Details can be seen in \cite{Pi1}. Let $f : M \to M$ be a Borel measurable map defined on a connected, compact, metric space $M$.
	
	\begin{definition}
		(Zooming contractions). A \textbf{\textit{zooming contraction}} is a sequence of functions $\alpha_{n}: [0,+\infty) \to [0,+\infty)$ such that
		
		\begin{itemize}
			\item $\alpha_{n}(r) < r, \text{for all} \, \, n \in \mathbb{N}, \text{for all} \, \, r>0.$
			
			\item $\alpha_{n}(r)<\alpha_{n}(s), \,\, if \,\, 0<r<s, \text{for all} \, \, n \in \mathbb{N}$.
			
			\item $\alpha_{m} \circ \alpha_{n}(r) \leq \alpha_{m+n}(r), \text{for all} \, \, r>0, \text{for all} \, \, m,n \in \mathbb{N}$.
			
			\item $\displaystyle \sup_{r \in (0,1)} \sum_{n=1}^{\infty}\alpha_{n}(r) < \infty$.
		\end{itemize}
		
	\end{definition}
	
	We have special types of zooming contractions. As defined in \cite{PV}, we call the contraction $(\alpha_{n})_{n}$ \textbf{\textit{exponential}} if $\alpha_{n}(r) = e^{-\lambda n} r$ for some $\lambda > 0$ and \textbf{\textit{Lipschitz}} if $\alpha_{n}(r) = a_{n} r$ with $0 \leq a_{n} < 1, a_{m}a_{n} \leq a_{m+n}$ and $\sum_{n=1}^{\infty} a_{n} < \infty$. In particular, every exponential contraction is Lipschitz. We can also have the example with $a_{n} = (n+b)^{-a}, a > 1, b>0$.

	\begin{definition}
		(Zooming times). Let $\alpha = (\alpha_{n})_{n}$ be a zooming contraction and $\delta>0$. We say that $n \in \mathbb{N}$ is an $(\alpha,\delta)$\textbf{\textit{-zooming time}} for $p \in X$ if there exists 
		a neighbourhood $V_{n}(p)$ of $p$ such that
		
		\begin{itemize}
			\item $f^{n}$ sends $\overline{V_{n}(p)}$ homeomorphically onto $\overline{B_{\delta}(f^{n}(p))}$;
			
			\item $d(f^{j}(x),f^{j}(y)) \leq \alpha_{n - j}(d(f^{n}(x),f^{n}(y)))$ for every $x,y \in V_{n}(p)$ and every $0 \leq j < n$.
		\end{itemize}
		
		We call $B_{\delta}(f^{n}(p))$ a \textbf{\textit{zooming ball}} and $V_{n}(p)$ a  \textbf{\textit{zooming pre-ball}}.
	\end{definition}
	
	We denote by $Z_{n}(\alpha,\delta,f)$ the set of points in $M$ for which $n$ is an $(\alpha, \delta)$- zooming time.
	
	\begin{definition}
		(Zooming measure) A $f$-non-singular finite measure $\mu$ defined on the Borel sets of M is called an $(\alpha,\delta)$-\textbf{\textit{weak zooming measure}} if $\mu$ almost every point has 
		infinitely many $(\alpha, \delta)$-zooming times. A weak zooming measure is called an $(\alpha,\delta)$-\textbf{\textit{zooming measure}} if
	\end{definition}
	\[
	\displaystyle d(x):= \limsup_{n \to \infty} \frac{1}{n} \{1 \leq j \leq n \mid x \in Z_{j}(\alpha,\delta,f)\} > 0,
	\]
	$\mu$-almost every $x \in M$.  
	
	\begin{definition}
		(Zooming set) Given an $(\alpha,\delta)$-zooming measure, we say that a forward invariant set $\Lambda \subset M$ such that $\mu(\Lambda)=1$ is an $(\alpha,\delta)$-\textbf{\textit{zooming set}} if $d(x) > 0$ holds for every $x \in \Lambda$. 
	\end{definition}
	\begin{remark}
		When it is clear which pair $(\alpha,\delta)$ we are taking, we call them simply a \textbf{\textit{zooming measure}} and a \textbf{\textit{zooming set}} 
	\end{remark}
	
	\begin{definition}
		(Bounded distortion) Given an $f$-non-singular finite measure $\mu$ defined on the Borel sets of $M$ with a jacobian $J_{\mu}f$, we say that the measure has \textbf{\textit{bounded distortion}} if there exists $\rho > 0$ such that
		\[
		\bigg{|}\log \frac{J_{\mu}f(y)}{J_{\mu}f(z)} \bigg{|} \leq \rho d(f^{n}(y),f^{n}(z)),
		\]
		$\mu$-almost every $y,z \in V_{n}(x)$, $\mu$-almost everywhere $x \in M$, for every zooming time $n$ of $x$.
	\end{definition}
	
	The map $f$ with an associated zooming measure is called a \textbf{\textit{zooming system}}. Every non-uniformly expanding map as considered in \cite{AOS}, for example, is a zooming system. Theorem C in \cite{Pi1} guarantees the existence of invariant probabilities for a zooming system when the reference zooming measure has bounded distortion. We denote the set of all $f$-invariant Borel probability measures as $\mathcal{M}_{f}^{1}(M)$. 
	
	Now, we state our main results. The following is similar to  Theorem A in \cite{BJ}. It is also somehow related to Theorem 11 in \cite{CLT}.
	
	\begin{theorema}\label{A}
		Let $f:M \to M$ be a continuous zooming system  with the zooming set $\Lambda$ dense on $M$. Given a $\beta$-Hölder continuous potential  $\phi:M \to \mathbb{R}$ such that
		 \[
		 \displaystyle m(\phi,f) : = \min_{\eta \in \mathcal{M}_{f}^{1}(M)} \bigg{\{}\int \phi d\eta\bigg{\}} \geq 0,
		 \]
		 then there exists a continuous function $\lambda_{0}: M \to \mathbb{R}$ such that
		 \[
		 \phi \geq \lambda_{0} - \lambda_{0} \circ f.
		 \] 
		 If $m(\phi,f) > 0$ then $\lambda_{0}$ can be taken $\gamma$-Hölder for some $\gamma$.
		 
		 In the case of all the integrals to be zero and the set periodic points dense in $M$, we obtain a version of the Livsic Theorem, that is, the function $\lambda_{0}$ can be taken such that
		 \[
		 \phi = \lambda_{0} - \lambda_{0} \circ f.
		 \]
	\end{theorema}
	
	The following result is similar to Proposition 10 in \cite{CLT}. Denote by $\mathcal{C}^{0}(M,\mathbb{R})$ the set of continuous functions $\phi: M \to \mathbb{R}$ with norm $\parallel \cdot \parallel_{\infty}$.
	
	\begin{theorema}\label{B}
		Let $\mathcal{K}$ be a compact convex subset of the set of probability measures on $M$ and $(\mathcal{H}, \parallel \cdot \parallel_{\mathcal{H}})$ be a dense Banach space in $\mathcal{C}^{0}(M, \mathbb{R})$ which embeds continuously in
		$\mathcal{C}^{0}(M, \mathbb{R})$. Then there exists a residual set $\mathcal{R}$ in $\mathcal{H}$ (for the  $\parallel \cdot \parallel_{\mathcal{H}}$-topology) such that, for
		all $\phi \in \mathcal{R}$, if
		\[
		m(\phi) := \max \bigg{\{} \int \phi d \eta \bigg{|} \eta \in \mathcal{K} \bigg{\}} \,\, \text{and} \,\, \mathcal{M}(\phi) : = \bigg{\{} \eta \in \mathcal{K} \bigg{|} \int \phi d \eta \bigg{|} \eta = m(\phi) \bigg{\}},
		\]
		then $\mathcal{M}(\phi)$ contains a unique measure.
	\end{theorema}

	\section{Proof of Theorem \ref{A}}
	
	Let $\mathcal{C}^{\beta}$ denote the space of $\beta$-Hölder continuous functions $\varphi: M \to \mathbb{R}$. We define the following norm on $\mathcal{C}^{\beta}$:
	\[
	\parallel \varphi \parallel_{\beta} : = \sup_{x,y \in M, x \neq y} \bigg{\{} \frac{|\varphi(x) - \varphi(y)|}{d(x,y)^{\beta}}\bigg{\}} < \infty.
	\]
   	
   	The following Lemma proves Theorem \ref{A}.
   	
	\begin{lemma}\label{preball}
		Given a $\beta$-Hölder function $\phi: M \to \mathbb{R}$ and $\mathcal{Z}$ the collection of all zooming pre-balls $V:= V_{n}(z)$, with zooming ball $B:= f^{n}(V) = B_{\delta}(f^{n}(z))$, there exists a continuous function $\lambda : M \to \mathbb{R}$ such that
		\[
		\phi  \geq \lambda \circ f - \lambda + m(\phi,f)  \,\, \text{i.e.} \,\, \phi \geq \overline{\lambda} - \overline{\lambda} \circ f,
		\]
		where $\overline{\lambda} = -\lambda$. As a consequence, if $m(\phi,f) > 0$, there exists a $\gamma$-Hölder function  $\lambda_{0} : M \to \mathbb{R}$ for some $\gamma$ such that $\phi \geq \lambda_{0} - \lambda_{0} \circ f$.
	\end{lemma}
	\begin{proof}
		We are assuming the zooming set $\Lambda$ dense in $M$, let $z \in \Lambda$ and $V_{n}(z) \in \mathcal{Z}$. We  consider $\varphi: = -\phi  + m(\phi,f)$. Taking $\theta < 1, \theta \approx 1$, define $\lambda : \Lambda \to \mathbb{R}$ as follows. 
		\[
		\lambda(x): = \sum_{i=0}^{\infty}\theta^{i}|\varphi(f^{i}(x))|.
		\]
		Given $x,y \in V_{n}(z)$, we have
		\[
		d(f^{i}(x),f^{i}(y)) \leq \alpha_{n-i}(d(f^{n}(x),f^{n}(y))), 0 \leq i \leq n
		\]
		So,
		\[
	    \sum_{i=0}^{n-1}\theta^{i}||\varphi(f^{i}(x))| - |\varphi(f^{i}(y))|| \leq  \sum_{i=0}^{n-1} \frac{||\varphi(f^{i}(x))| - |\varphi(f^{i}(y))||}{d(f^{i}(x),f^{i}(y))^{\beta}} \theta^{i}d(f^{i}(x),f^{i}(y))^{\beta} \leq 
	    \]
	    \[
	    \sum_{i=0}^{n-1} \parallel \varphi \parallel_{\beta}  \theta^{i}\alpha_{n-i}(d(f^{n}(x),f^{n}(y)))^{\beta} 
		< \parallel \varphi \parallel_{\beta}\sum_{i=0}^{n-1}  \theta^{i}d(f^{n}(x),f^{n}(y))^{\beta}.
		\]
		We used the property of zooming contractions that $\alpha_{n-i}(r)< r,  r > 0$. It holds that
		\[
		|\lambda(x) - \lambda(y)|=\bigg{|}\sum_{i=0}^{\infty}\theta^{i}|\varphi(x)| - \sum_{i=0}^{\infty}\theta^{i}|\varphi(y)|\bigg{|} \leq  \sum_{i=0}^{\infty}\theta^{i}||\varphi(f^{i}(x))| - |\varphi(f^{i}(y))|| = 
		\]
		\[
		\sum_{i=0}^{n-1}\theta^{i}||\varphi(f^{i}(x))| - |\varphi(f^{i}(y))|| + \sum_{i=n}^{\infty}\theta^{i}||\varphi(f^{i}(x))| - |\varphi(f^{i}(y))|| \leq 
		\]
		\[
		\parallel \varphi \parallel_{\beta}\sum_{i=0}^{
			n - 1}  \theta^{i}  d(f^{n}(x),f^{n}(y))^{\beta} + \sum_{i=n}^{\infty}\theta^{i}||\varphi(f^{i}(x))| - |\varphi(f^{i}(y))||
		\]
		Given $\epsilon > 0$ we can find $n_{0}$ such that 
		\[
		\sum_{i=n_{0}}^{\infty}\theta^{i}||\varphi(f^{i}(x))| - |\varphi(f^{i}(y))|| < \epsilon/2, 
		\]
		for every $x,y \in M$. Also, by uniform continuity of $f$, we can obtain
		\[
		\parallel \varphi \parallel_{\beta}\sum_{i=0}^{n_{0} - 1}  \theta^{i}  d(f^{n_{0}}(x),f^{n_{0}}(y))^{\beta} < \epsilon/2, 
		\]
		if $x$ and $y$ are close enough. It implies that $\lambda$ is continuous. Taking $\theta$ close enough to $1$, we have that $\lambda + \delta \geq \varphi + \lambda \circ f$ for $m(\phi,f) > \delta > 0$ because
		\[
		\lambda(f(x)) - \lambda(x) = \sum_{i=0}^{\infty}\theta^{i}|\varphi(f^{i}(f(x)))| - \sum_{i=0}^{\infty}\theta^{i}|\varphi(f^{i}(x))| = (1/\theta)\sum_{i=0}^{\infty}\theta^{i + 1}|\varphi(f^{i+1}(x))| - \sum_{i=0}^{\infty}\theta^{i}|\varphi(f^{i}(x))| = 
		\]
		\[
		(1/\theta)\sum_{i=1}^{\infty}\theta^{i}|\varphi(f^{i}(x))| - \sum_{i=0}^{\infty}\theta^{i}|\varphi(f^{i}(x))| = ((1/\theta)-1)\sum_{i=1}^{\infty}\theta^{i}|\varphi(f^{i}(x))| - |\varphi(x)| \leq -\varphi(x) + \delta
		\] 
		which implies that
		\[
		\varphi = -\phi + m(\phi,f) \leq  \lambda - \lambda 
		\circ f + \delta \implies \phi \geq \phi - m(\phi,f) + \delta \geq (-\lambda) - (-\lambda) \circ f,
		\]
		and we obtain $\phi \geq \overline{\lambda} - \overline{\lambda} \circ f$. The function $\overline{\lambda}$ is uniformly continuous on the dense set $\Lambda$ and we can extend it to the whole compact space $M$ and still have the inequality $\phi \geq \overline{\lambda} - \overline{\lambda} \circ f$. Moreover, if $m(\phi,f) > 0$, by denseness of the Hölder functions, there exists a $\gamma$-Hölder function $\lambda_{0} : M \to \mathbb{R}$ for some $\gamma$ such that for $\epsilon > 0$ small enough we have $\parallel \overline{\lambda} - \lambda_{0} \parallel_{\infty} < \epsilon$ and $\phi > \overline{\lambda} - \overline{\lambda} \circ f + 2\epsilon \geq \lambda_{0} - \lambda_{0} \circ f$. The Lemma is proved.
	\end{proof}
	
	We readily obtain the following corollary as a weak version of Livsic Theorem.
	
	\begin{corollary}
		Let $f:M \to M$ be a continuous zooming system with the zooming set $\Lambda$ dense on $M$. Given a $\beta$-Hölder continuous potential  $\phi:M \to \mathbb{R}$ such that 
		\[
		\int \phi d \eta = 0, \,\, \text{for every} \,\, \eta \in \mathcal{M}_{f}^{1}(M),
		\]
		then there exist continuous functions $\lambda_{1},\lambda_{2}: M \to \mathbb{R}$ such that
		\[
		\lambda_{1} - \lambda_{1} \circ f \leq \phi \leq \lambda_{2} - \lambda_{2} \circ f.
		\] 
	\end{corollary}
	\begin{proof}
		In fact, we have both $\int \phi d \eta \geq 0$ and $\int -\phi d \eta \geq 0$ for all $\eta \in \mathcal{M}_{f}^{1}(M)$. By applying Theorem \ref{A} for both $\phi$ and $-\phi$, we obtain the result.
	\end{proof}

\begin{corollary}[Livsic Theorem]\label{inequality}
	Let $f:M \to M$ be a continuous zooming system with the zooming set $\Lambda$ dense in $M$ and the set of periodic points dense in $M$. Given a $\beta$-Hölder continuous potential  $\phi:M \to \mathbb{R}$ such that 
	\[
	\int \phi d \eta = 0, \,\, \text{for every} \,\, \eta \in \mathcal{M}_{f}^{1}(M),
	\]
	then there exists a continuous function $\lambda:  M \to \mathbb{R}$ such that
	\[
	\phi = \lambda - \lambda \circ f.
	\] 
\end{corollary}
\begin{proof}
	The Corollary \ref{inequality} gives continuous functions $\lambda_{1},\lambda_{2} : M \to \mathbb{R}$ such that
	\[
	\lambda_{1} - \lambda_{1} \circ f \leq \phi \leq \lambda_{2} - \lambda_{2} \circ f.
	\]
	It implies that
	\[
	(\lambda_{2} - \lambda_{1})\circ f \leq \lambda_{2} - \lambda_{1}.
	\]
	Setting $\lambda_{0} := \lambda_{2} - \lambda_{1}$ and given $p \in M$ a periodic point with period $k \in \mathbb{N}$, we have
	\[
	\lambda_{0}(p) = \lambda_{0}(f^{k}(p)) \leq \lambda_{0}(f^{k-1}(p)) \leq \dots \leq \lambda_{0}(f(p)) \leq \lambda_{0}(p) \implies \lambda_{0}(f(p)) = \lambda_{0}(p).
	\]
	Once $\lambda_{0}$ is continuous and the set of periodic points is dense in $M$, we have that
	\[
	\lambda_{0} \circ f = \lambda_{0} \implies \lambda_{1} - \lambda_{1} \circ f = \phi = \lambda_{2} - \lambda_{2} \circ f.
	\]
	Taking either $\lambda = \lambda_{1}$ or $\lambda = \lambda_{2}$ we obtain
	\[
	\phi = \lambda - \lambda \circ f.
	\] 
\end{proof}
	
	This concludes the proof of Theorem \ref{A}.
	
	\section{Proof of Theorem \ref{B}}
	
	We follow the ideas of Section 2 in \cite{CLT}. We begin by proving that, generically, there exists a unique maximizing measure. This comes mainly from the fact that, for a compact convex set in $\mathbb{R}^{n}$ , among the set of hyperplanes which support the convex set, the set of those hyperplanes having an intersection reduced to a single point is generic (intersection of countably many open and dense sets). Nonetheless, the proof has to be carried in infinite dimension and requires more details.
	
	We first recall some definitions. We say that a point $p$ is an extremal point of a compact convex set $C$ of $\mathbb{R}^{n}$ if $p$ is not the mid point of a segment totally included in $C$. We say that $p$ is strictly extremal if there exists a linear form which attains its maximum at the point $p$ only. A classical result (see \cite{R}) states that $C$ is equal to the closed convex hull of its strictly extremal points.  We reproduce now the proof of Proposition 10 in \cite{CLT}.
	
	\begin{proof}
		Let $\{\phi_{n}\}_{n \geq 1}$ be a dense subset of the unit ball of $\mathcal{H}$.	Since $\mathcal{H}$ is dense,
		\[
		d(\eta,\eta') : = \sum_{n=1}^{\infty}\frac{1}{2^{n}} \bigg{|} \int \phi_{n} d \eta - \int \phi_{n} d \eta' \bigg{|}
		\]
		defines a metric on $\mathcal{K}$ compatible with the weak* topology. Let us call
		\[
		\mathcal{R}_{\epsilon} :  = \{\phi \in \mathcal{C}^{0}(M,\mathbb{R}) \mid \text{diam}\mathcal{M}(\phi) < \epsilon\}.
		\]
		We claim that $\mathcal{R}_{\epsilon}$ is open in $\mathcal{C}^{0}(M,\mathbb{R})$ and $\mathcal{R}_{\epsilon} \cap \mathcal{H}$ is dense in $\mathcal{H}$ for the $\parallel \cdot \parallel_{\mathcal{H}}$-topology.  The desired residual set will be $\mathcal{R} = \cap_{n \geq 1} \mathcal{R}_{1/n} \cap \mathcal{H}$.
		
		We show by contradiction that $\mathcal{R}_{\epsilon}$ is open. If not, one can find $\phi \in \mathcal{R}_{\epsilon}$, $\varphi_{n} \in  \mathcal{C}^{0}(M,\mathbb{R})$ and $\mu_{n}, \nu_{n} \in \mathcal{M}(\phi + \varphi_{n})$ such that $\parallel \varphi_{n} \parallel_{\infty}$ converges to zero and $d(\mu_{n},\nu_{n}) \geq \epsilon$ for all $n$. We may assume by taking a subsequence that $\mu_{n} \to \mu_{0}$ and  $\nu_{n} \to \nu_{0}$. Let us prove that $\mu_{0} \in \mathcal{M}(\phi)$: indeed for every $\mu \in \mathcal{K}$,
		\[
		\int (\phi + \varphi_{n}) d\mu \leq \int (\phi + \varphi_{n}) d\mu_{n} \leq \int \phi d\mu_{n} + \parallel \varphi_{n} \parallel_{\infty}
		\]
		and $\int \phi d\mu \leq \int \phi d\mu_{0}$, by taking the limit on $n$. For the same reason we have that $\nu_{0} \in \mathcal{M}(\phi)$. We have obtained the contradiction once we have $d(\mu_{0},\nu_{0}) \geq \epsilon$.
		
		We now show that $\mathcal{R}_{\epsilon} \cap \mathcal{H}$ is dense in $\mathcal{H}$. Let $\phi_{0} \in \mathcal{H}$ and $\mathcal{K}_{0} = \mathcal{M}(\phi_{0})$. The continuous projection $\pi_{n} : \mathcal{K} \to \mathbb{R}^{n}$, $\pi_{n}(\mu) = \big{(} \int \phi_{1} d\mu, \dots, \int \phi_{n} d\mu  \big{)}$ sends $\mathcal{K}_{0}$ to a compact convex set $\pi_{n}(\mathcal{K}_{0})$ which admits a stictly extremal point $p_{n}$. We first notice that $\text{diam}\pi_{n}^{-1}(p_{n}) < 2^{-n}$ and we choose $n$ large enough so that $2^{-n} < \epsilon$. By definition of $p_{n} = (p^{1}, \dots, p^{n})$ there exists $(a^{1}, \dots, a^{n}) \in \mathcal{R}^{n}$ such that
		\[
		\sum_{i=1}^{n} a^{i}p^{i} > \sum_{i=1}^{n} a^{i}q^{i} \,\, \text{for all} \,\, q = (q^{1}, \dots, q^{n}) \in \pi_{n}(\mathcal{K}_{0}), q \neq p.
		\]
		In particular, if $\psi = \sum_{i=1}^{n} a^{i}\phi_{i}$,
		\[
		m_{0}(\psi) := \max \bigg{\{} \int \psi d \eta \bigg{|} \eta \in \mathcal{K}_{0} \bigg{\}} \,\, \text{and} \,\, \mathcal{M}_{0}(\psi) : = \bigg{\{} \eta \in \mathcal{K}_{0} \bigg{|} \int \psi d \eta \bigg{|} \eta = m(\psi) \bigg{\}},
		\]
		then $\mathcal{M}_{0}(\psi) = \pi_{n}^{-1}(p_{n})$ has diameter less than $\epsilon$. We show that for small enough $\delta > 0$, $\psi_{\delta} = (1 - \delta)\varphi + \delta \psi \in \mathcal{R}_{\epsilon}$. More precisely we show that, for any open set $\mathcal{U} \supset \mathcal{M}_{0}(\psi)$, for any sufficiently small $\delta > 0$ we have $\mathcal{U} \supset \mathcal{M}(\psi_{\delta})$. By contradiction, there exists a sequence $\mu_{n} \in \mathcal{M}(\psi_{\delta_{n}}) \backslash \mathcal{U}$ for some $\delta_{n} \to 0$. We may assume that $\mu_{n} \to \mu_{0} \in \mathcal{K} \backslash \mathcal{U}$. We first show that $\mu_{0} \in \mathcal{K}_{0} = \mathcal{M}(\varphi)$: for every $\mu \in \mathcal{K}$,
		\[
		\int \psi_{\delta_{n}} d\mu \leq \int \psi_{\delta_{n}} d\mu_{n} \leq \int \varphi d\mu_{n} + \delta_{n}\parallel \psi - \varphi\parallel_{\infty}, 
		\] 
		and by taking limit in $n$, $\int \varphi d\mu \leq \int \varphi d\mu_{0}$. We then show that $\mu_{0} \in \mathcal{M}_{0}(\psi)$: for every $\mu \in \mathcal{K}_{0}$,
		\[
		\int \psi_{\delta_{n}} d\mu = (1 - \delta_{n}) \int \varphi d\mu + \delta_{n} \int \psi d\mu \leq (1 - \delta_{n}) \int \varphi d\mu_{n} + \delta_{n} \int \psi d\mu_{n}.
		\]
		Since $\int \varphi d\mu_{n} \leq \int \varphi d\mu$, we have obtained $\int \psi d\mu \leq \int \psi d\mu_{n}$ and at the limit $\int \psi d\mu \leq \int \psi d\mu_{0}$. We have obtained a contradiction since $\mu_{0} \not \in \mathcal{U}$.
	\end{proof}
	
	Analogously to Theorem 6 in \cite{CLT} we obtain the next corollary.
	
	\begin{corollary}
		Let $f:M \to M$ be a continuous zooming system with the zooming set $\Lambda$ dense in $M$. Then the set of $\beta$-Hölder functions $\phi$ admitting a unique maximizing measure is generic in $\mathcal{C}^{\beta}$. For such functions $\phi$, the map $f$ is strictly ergodic on the support of its unique maximizing measure.
	\end{corollary}
	\begin{proof}
		This is a direct consequence of Theorem \ref{B} by taking $\mathcal{K} = \mathcal{M}_{f}^{1}(M)$ and $\mathcal{H} = \mathcal{C}^{\beta}$, once $\mathcal{K} = \mathcal{M}_{f}^{1}(M)$ is a compact and convex set of probability measures and $\mathcal{C}^{\beta}$ is dense in $\mathcal{C}^{0}(M,\mathbb{R})$ for any compact metric space $M$.
	\end{proof}
	
	\section{Applications}
	
	In this section, we give examples of zooming systems. Most of them are nonuniformly expanding maps in the sense of \cite{A} or \cite{AOS}. The reference \cite{Pi1} contains a more general approach. We begin by the hyperbolic times as a particular case of zooming times.
	
		\subsection{Hyperbolic Times}
	
	The idea of hyperbolic times is a key notion on the study of non-uniformly hyperbolic dynamics and it was introduced by Alves et al. 
	This is powerful to get expansion in the context of non-uniform expansion. Here, we recall the basic definitions and results on hyperbolic times that we will use later on. In the following, we give definitions taken from \cite{A} and \cite{Pi1}.
	
	\begin{definition}
		Let $M$ be a compact Riemannian manifold of dimension $d \geq 1$ and $f:M \to M$ a continuous map defined on $M$.
		The map $f$ is called \textbf{\textit{non-flat}} if it is a local $C^{1 + \alpha}, (\alpha >0)$ diffeomorphism in the whole manifold except in a 
		non-degenerate set $\mathcal{C} \subset M$. We say that $\mathcal{C} \subset M$ is a \textbf{\textit{non-degenerate set}}
		if there exist $\beta, B > 0$ such that the following two conditions hold.
		
		\begin{itemize}
			
			\item $\frac{1}{B} d(x,\mathcal{C})^{\beta} \leq \frac{\parallel Df(x) v\parallel}{\parallel v \parallel} \leq B d(x,\mathcal{C})^{-\beta}$ for all $v \in T_{x}M$.
			
			For every $x, y \in M \backslash \mathcal{C}$ with $d(x,y) < d(x,\mathcal{C})/2$ we have
			
			\item $\mid \log \parallel Df(x)^{-1} \parallel - \log \parallel Df(y)^{-1} \parallel \mid \leq \frac{B}{d(x,\mathcal{C})^{\beta}} d(x,y)$.
			
		\end{itemize}
		
	\end{definition}
	
	In the following, we give the definition of a hyperbolic time \cite{A}, \cite{Pi1}.
	
	\begin{definition}
		(Hyperbolic times). Let us fix $0 < b = \frac{1}{3} \min\{1,1 \slash \beta\} < \frac{1}{2} \min\{1,1\slash \beta\}$. 
		Given $0 < \sigma < 1$ and $\epsilon > 0$, we will say that $n$ is a $(\sigma, \epsilon)$\textbf{\textit{-hyperbolic time}} for a point $x \in M$ 
		(with respect to the non-flat map $f$ with a $\beta$-non-degenerate critical/singular set $\mathcal{C})$ if for all $1 \leq k \leq n$ we have 
		\[
		\prod_{j=n-k}^{n-1} \|(Df \circ f^{j}(x)^{-1}\| \leq \sigma^{k} \,\, \text{and} \,\, dist_{\epsilon}(f^{n-k}(x), \mathcal{C}) \geq \sigma^{bk}.
		\]
		We denote de set of points of $M$ such that $n \in \mathbb{N}$ is a $(\sigma,\epsilon)$-hyperbolic time by $H_{n}(\sigma,\epsilon,f)$. 
	\end{definition}
	
	\begin{proposition}
		(Positive frequence). Given $\lambda > 0$ there exist $\theta > 0$ and $\epsilon_{0} > 0$ such that, for every $x \in U$ and $\epsilon \in (0,\epsilon_{0}]$,
		\[
		\#\{1 \leq j \leq n; \,\, x \in H_{j}(e^{-\lambda \slash 4}, \epsilon, f) \} \geq \theta n,
		\]
		whenever $\frac{1}{n}\sum_{i=0}^{n-1}\log\|(Df(f^{i}(x)))^{-1}\|^{-1} \geq \lambda$ and $\frac{1}{n}\sum_{i=0}^{n-1}-\log dist_{\epsilon}(x, \mathcal{C}) \leq \frac{\lambda}{16 \beta}$.
	\end{proposition}
	
	If $f$ is non-uniformly expanding, it follows from the proposition that the points of $U$ have infinitely many moments with positive frequency of hyperbolic times. In particular, they have infinitely many hyperbolic times.
	
	The following proposition shows that the hyperbolic times are indeed zooming times, where the zooming contraction is $\alpha_{k}(r) = \sigma^{k/2}r$.
	
	\begin{proposition}
		Given $\sigma \in (0,1)$ and $\epsilon > 0$, there is $\delta,\rho > 0$, depending only on $\sigma$ and $\epsilon$ and on the map $f$, such that if $x \in H_{n}(\sigma,\epsilon,f)$ then there exists a neighbourhood $V_{n}(x)$ of $x$ with the following properties for all $y,z \in V_{n}(x)$:
		
		\begin{enumerate}
			\item[(1)] $f^{n}$ maps $\overline{V_{n}(x)}$ diffeomorphically onto the ball $\overline{B_{\delta}(f^{n}(x))}$;
			\item[(2)] $dist(f^{n-j}(y),f^{n-j}(z)) \leq \sigma^{j\slash 2} dist(f^{n}(y), f^{n}(z)), \text{for all} \, \, y,z \in V_ {n}(x)$ and $1 \leq j < n$.
			\item[(3)]$\log \frac{\mid \det Df^{n}(y)\mid}{\mid \det Df^{n}(z)\mid} \leq \rho d(f^{n}(y),f^{n}(z))$.
		\end{enumerate}
	
	\end{proposition}
	
	The sets $V_{n}(x)$ are called hyperbolic pre-balls and their images $f^{n}(V_{n}(x)) = B_{\delta}(f^{n}(x))$, hyperbolic balls.
	

	\subsection{Uniformly Expanding Maps} As can be seen in \cite{OV} Chapter 11, we have the so-called \textit{uniformly expanding maps} which is defined on a compact differentiable manifold $M$ as a $C^{1}$ map $f:M \to M$ (with no critical set) for which there exists $\sigma > 1$ such that
	\[
	\|Df(x)v\|\geq \sigma \|v\|, \,\, \text{for every} \,\, x \in M, v \in T_{x}M.
	\]
	
	For compact metric spaces $(M,d)$ we define it as a continuous map $f:M \to M$, for which there exists $\sigma > 1, \delta>0$ such that for every $x \in M$ we have that the image of the ball $B(x,\delta)$ contains a neighbourhood of the ball $B(f(x),\delta)$ and
	\[
	d(f(a),f(b)) \geq \sigma d(a,b), \,\, \text{for every} \,\, a,b \in B(x,\delta).
	\]
	We observe that the uniformly expanding maps on differentiable manifolds satisfy the conditions for the definition on compact metric spaces, when they are seen as Riemannian manifolds. If the metric space $M$ is connected, uniform expansion implies topological exactness. The zooming set $\Lambda$ is the whole space $M$.
	
	\subsection{Local Diffeomorphisms} As can be seen in details in \cite{A}, we will briefly describe a class of non-uniformly expanding maps.
	
	Here we present a robust ($C^{1}$ open) classes of local diffeomorphisms (with no critical set) that are non-uniformly expanding. Such classes of maps can be obtained, e.g., through deformation of a uniformly expanding map by isotopy inside some small region. In general, these maps are not uniformly expanding: deformation can be made in such way that the new map has periodic saddles.
	
	Let $M$ be a compact manifold supporting some uniformly expanding map $f_{0}$. $M$ could be the $d$-dimensional torus $\mathbb{T}^{d}$, for instance. Let $V \subset M$ be some small compact domain, so that the restriction of $f_{0}$ to $V$ is injective. Let $f$ be any map in a sufficiently small $C^{1}$-neighbourhood $\mathcal{N}$ of $f_{0}$ so that:
	
	\begin{itemize}
		\item $f$ is \textit{volume expanding everywhere}: there exists $\sigma_{1} > 1$ such that
		\[
		|\det Df(x)| > \sigma_{1} \,\, \text{for every} \,\, x \in M;
		\]	
		
		\item $f$ is \textit{expanding outside} $V$: there exists $\sigma_{0} > 1$ such that
		\[
		\|Df(x)^{-1}\| < \sigma_{0} \,\, \text{for every} \,\, x \in M \backslash V;
		\] 
		
		\item $f$ is \textit{not too contracting on} $V$: there is some small $\delta > 0$ such that
		\[
		\|Df(x)^{-1}\| < 1 + \delta \,\, \text{for every} \,\, x \in V.
		\]
	\end{itemize}
	
	In \cite{A} it is shown that this class satisfy the condition for non-uniform expansion. We have here the zooming set $\Lambda$ dense in $M$.

	Now, we recall examples given in \cite{AOS}, where the expanding set is dense in $M$ and there exist critical points. 
	
	\subsection{Viana maps} We recall the definition of the open class of maps with critical sets in dimension 2, introduced by M. Viana in \cite{V}. We skip the technical
	points. It can be generalized for any dimension (See \cite{A}).

	Let $a_{0} \in (1,2)$ be such that the critical point $x=0$ is pre-periodic for the quadratic map $Q(x)=a_{0} - x^{2}$. Let $S^{1}=\mathbb{R}/\mathbb{Z}$ and 
	$b:S^{1} \to \mathbb{R}$ a Morse function, for instance $b(\theta) = \sin(2\pi\theta)$. For fixed small $\alpha > 0$, consider the map
	\[
	\begin{array}{c}
		f_{0}: S^{1} \times \mathbb{R} \longrightarrow S^{1} \times \mathbb{R}\\
		\,\,\,\,\,\,\,\,\,\,\,\,\,\,\,\,\,\,\,\ (\theta,x) \longmapsto (g(\theta),q(\theta,x))
	\end{array}
	\] 
	where $g$ is the uniformly expanding map of the circle defined by $g(\theta)=d\theta
	(mod\mathbb{Z})$ for some $d \geq 16$, and $q(\theta,x) = a(\theta) - x^{2}$ with $a(\theta) = a_{0} + \alpha b(\theta)$. It is easy to check that for $\alpha > 0$ 
	small enough there is an interval $I \subset (-2,2)$ for which $f_{0}(S^{1} \times I)$ is contained in the interior of $S^{1} \times I$, as can be seen in \cite{V}.      Thus, any map $f$ sufficiently
	close to $f_{0}$ in the $C^{0}$ topology has $S^{1} \times I$ as a forward invariant region. We consider from here on these maps $f$ close to $f_{0}$ restricted to 
	$S^{1} \times I$. Taking into account the expression of $f_{0}$ it is not difficult to check that for $f_{0}$ (and any map $f$ close to $f_{0}$ in the $C^{2}$ topology)
	the critical set is non-degenerate. Also can be seen in \cite{V}.
	
	The main properties of $f$ in a $C^{3}$ neighbourhood of $f$ that we will use here are summarized below (See \cite{A},\cite{AV},\cite{Pi1}):
	
	\begin{enumerate}
		\item[(1)] $f$ is \textbf{\textit{non-uniformly expanding}}, that is, there exist $\lambda > 0$ and a Lebesgue full measure set $H \subset S^{1} \times I$ such that 
		for all point $p=(\theta, x) \in H$, the following holds
		\[
		\displaystyle \limsup_{n \to \infty} \frac{1}{n} \sum_{i=0}^{n-1} \log \parallel Df(f^{i}(p))^{-1}\parallel^{-1} < -\lambda.
		\]  
		\item[(2)] Its orbits have \textbf{\textit{slow approximation to the critical set}}, that is, for every $\epsilon > 0$ the exists $\delta > 0$ such that for every point
		$p=(\theta, x) \in H \subset S^{1} \times I$, the following holds 
		\[
		\displaystyle \limsup_{n \to \infty} \frac{1}{n} \sum_{i=0}^{n-1} - \log \text{dist}_{\delta}(p,\mathcal{C}) < \epsilon.
		\]  
		where 
		\[
		\text{dist}_{\delta}(p,\mathcal{C}) =  
		\left\{ 
		\begin{array}{ccc}
			dist(p,\mathcal{C}), & if & dist(p,\mathcal{C}) < \delta\\
			1 & if & dist(p,\mathcal{C}) \geq \delta 
		\end{array}
		\right.
		\]  
		\item[(3)] $f$ is topologically exact;
		
		\item[(4)] $f$ is strongly topologically transitive;
		
		\item[(5)] it has a unique ergodic absolutely continuous invariant (thus SRB) measure;
		
		\item[(6)]the density of the SRB measure varies continuously in the $L^{1}$ norm with $f$.
	\end{enumerate}
	
	\begin{remark}
		We observe that this definition of non-uniformly expanding map implies positive frequency of hyperbolic times as can be seen in \cite{A}. So, it is included in our definition of a zooming system. Also, the zooming set is dense in $M = S^{1} \times I$.
	\end{remark}

	\subsection{Benedicks-Carleson Maps} We study a class of non-hyperbolic maps of the interval with the condition of exponential growth of the derivative at critical values, called 
	\textbf{\textit{Collet-Eckmann Condition}}. We also ask the map to be $C^{2}$ and topologically exact and the critical points to have critical order 
	$2 \leq \alpha < \infty$.
	
	Given a critical point $c \in I$, the \textbf{\textit{critical order}} of $c$ is a number $\alpha_{c} > 0$ such that 
	$f(x) = f(c) \pm |g_{c}(x)| ^{\alpha_{c}}, \,\, \text{for all} \, \, x \in \mathcal{U}_{c}$ where $g_{c}$ is a diffeomorphism 
	$g_{c}: \mathcal{U}_{c} \to g(\mathcal{U}_{c})$ and $\mathcal{U}_{c}$ is a neighbourhood of $c$. 
	
	Let $\delta>0$ and denote $\mathcal{C}$ the set of critical points and $\displaystyle B_{\delta} = \cup_{c \in \mathcal{C}} (c - \delta, c + \delta)$. 
	Given $x \in I$, we suppose that
	
	\begin{itemize}
		\item \textbf{(Expansion outside $B_{\delta}$)}.  There exists $\kappa > 1 $ and $\beta > 0$ such that, if $x_{k} = f^{k}(x) \not \in B_{\delta}, \,\, 0 \leq k \leq n-1$ then $|Df^{n}(x)| \geq \kappa \delta^{(\alpha_{\max} -1)}e^{\beta n}$, where $\alpha_{\max} = \max \{\alpha_{c}, c \in \mathcal{C}\}$. Moreover, if $x_{0} \in f(B_{\delta})$ or $x_{n} \in B_{\delta}$ then $|Df^{n}(x)| \geq \kappa e^{\beta n}$.
		
		\item \textbf{(Collet-Eckmann Condition)}. There exists $\lambda > 0$ such that 
		\[
		|Df^{n}(f(c))| \geq e^{\lambda n}.
		\]
		
		\item \textbf{(Slow Recurrence to $\mathcal{C}$)}. There exists $\sigma \in (0, \lambda/5)$ such that 
		\[
		dist(f^{k}(x), \mathcal{C}) \geq e^{-\sigma k}.
		\]
	\end{itemize}

The above conditions has an important contribuition by Freitas in (\cite{F}). This is an important class of non-uniformly expanding maps.

	\subsection{Rovella Maps}
	
	There is a class of non-uniformly expanding maps known as \textbf{\textit{Rovella Maps}}. They are derived from the so-called \textit{Rovella Attractor},
	a variation of the \textit{Lorenz Attractor}. We proceed with a brief presentation. See \cite{AS} for details.
	
	\subsubsection{Contracting Lorenz Attractor}
	
	The geometric Lorenz attractor is the first example of a robust attractor for a flow containing a hyperbolic singularity. The attractor is a transitive maximal invariant
	set for a flow in three-dimensional space induced by a vector field having a singularity at the origin for which the derivative of the vector field at the singularity has
	real eigenvalues $\lambda_{2} < \lambda_{3} < 0 < \lambda_{1}$ with $\lambda_{1} + \lambda_{3} > 0$. The singularity is accumulated by regular orbits which prevent the 
	attractor from being hyperbolic.
	
	The geometric construction of the contracting Lorenz attractor (Rovella attractor) is the same as the geometric Lorenz attractor. The only difference is the condition
	(A1)(i) below that gives in particular $\lambda_{1} + \lambda_{3} < 0$. The initial smooth vector field $X_{0}$ in $\mathbb{R}^{3}$ has the following properties:
	
	\begin{itemize}
		
		\item[(A1)] $X_{0}$ has a singularity at $0$ for which the eigenvalues $\lambda_{1},\lambda_{2},\lambda_{3} \in \mathbb{R}$ of $DX_{0}(0)$ satisfy:
		\begin{itemize}
			
			\item[(i)] $0 < \lambda_{1} < -\lambda_{3} < 0 < -\lambda_{2}$,
			
			\item[(ii)] $r > s+3$, where $r=-\lambda_{2}/\lambda_{1}, s=-\lambda_{3}/\lambda_{1}$;
		\end{itemize}
		
		\item[(A2)] there is an open set $U \subset \mathbb{R}^{3}$, which is positively invariant under the flow, containing the cube
		$\{(x,y,z) : \mid x \mid \leq 1, \mid y \mid \leq 1, \mid x \mid \leq 1\}$ and supporting the \textit{Rovella attractor}
		\[
		\displaystyle \Lambda_{0} = \bigcap_{t \geq 0} X_{0}^{t}(U).
		\]
		
		The top of the cube is a Poincar\'e section foliated by stable lines $\{x = \text{const}\} \cap \Sigma$ which are invariant under Poincar\'e first return map $P_{0}$.
		The invariance of this foliation uniquely defines a one-dimensional map $f_{0} : I \backslash \{0\} \to I$ for which
		\[
		f_{0} \circ \pi = \pi \circ P_{0},
		\]
		where $I$ is the interval $[-1,1]$ and $\pi$ is the canonical projection $(x,y,z) \mapsto x$;
		
		\item[(A3)] there is a small number $\rho >0$ such that the contraction along the invariant foliation of lines $x =$const in $U$ is stronger than $\rho$.
	\end{itemize}
	
	See \cite{AS} for properties of the map $f_{0}$.
	
	\subsubsection{Rovella Parameters}
	
	The Rovella attractor is not robust. However, the chaotic attractor persists in a measure theoretical sense: there exists a one-parameter family of positive Lebesgue measure
	of $C^{3}$ close vector fields to $X_{0}$ which have a transitive non-hyperbolic attractor. In the proof of that result, Rovella showed that there is a set of parameters
	$E \subset (0,a_{0})$ (that we call \textit{Rovella parameters}) with $a_{0}$ close to $0$ and $0$ a full density point of $E$, i.e.
	\[
	\displaystyle \lim_{a \to 0} \frac{\mid E \cap (0,a) \mid}{a} = 1,
	\]
	such that:
	
	\begin{itemize}
		\item[(C1)] there is $K_{1}, K_{2} > 0$ such that for all $a \in E$ and $x \in I$
		\[
		K_{2} \mid x \mid^{s-1} \leq f_{a}'(x) \leq K_{1} \mid x \mid^{s-1},
		\]
		where $s=s(a)$. To simplify, we shall assume $s$ fixed.
		
		\item[(C2)] there is $\lambda_{c} > 1$ such that for all $a \in E$, the points $1$ and $-1$ have \textit{Lyapunov exponents} greater than $\lambda_{c}$:
		\[
		(f_{a}^{n})'(\pm 1) > \lambda_{c}^{n}, \,\, \text{for all} \, \, n \geq 0;
		\]
		
		\item[(C3)] there is $\alpha > 0$ such that for all $a \in E$ the \textit{basic assumption} holds:
		\[
		\mid f_{a}^{n-1}(\pm 1)\mid > e^{-\alpha n}, \,\, \text{for all} \, \, n \geq 1;
		\]
		
		\item[(C4)] the forward orbits of the points $\pm 1$ under $f_{a}$ are dense in $[-1,1]$ for all $a \in E$.
		
	\end{itemize}
	
	\begin{definition}
		We say that a map $f_{a}$ with $a \in E$ is a \textbf{\textit{Rovella Map}}. 
	\end{definition}
	
	\begin{theorem}
		(Alves-Soufi \cite{AS}) Every Rovella map is non-uniformly expanding and has slow recurrence to the critical set. 
	\end{theorem}

	In the following, we give definitions for a map on a metric space to have similar behaviour to maps with hyperbolic times and which can be found in \cite{Pi1}.  
	
	Given $M$ a metric spaces and $f: M \to M$, we define for $p \in M$:
	\[
	\displaystyle \mathbb{D}^{-}(p) = \liminf \frac{d(f(x),f(p)}{d(x,p)}
	\]
	Define also,
	\[
	\displaystyle \mathbb{D}^{+}(p) = \limsup \frac{d(f(x),f(p)}{d(x,p)}
	\]
	We will consider points $x \in M$ such that 
	\[
	\displaystyle \limsup_{n \to \infty} \frac{1}{n} \sum_{i=0}^{n-1} \log \mathbb{D}^{-} \circ f^{i}(x) > 0.
	\]  
	The critical set is the set of points $x \in M$ such that $\mathbb{D}^{-}(x) = 0$ or $\mathbb{D}^{+}(x) = \infty$. For  the non-degenerateness we ask that there exist $B, \beta >0$ such that
	
	\begin{itemize}
		
		\item $\frac{1}{B} d(x,\mathcal{C})^{\beta} \leq \mathbb{D}^{-}(x) \leq \mathbb{D}^{+}(x) \leq B d(x,\mathcal{C})^{-\beta}$.
		
		For every $x, y \in M \backslash \mathcal{C}$ with $d(x,y) < d(x,\mathcal{C})/2$ we have
		
		\item $\mid \log \mathbb{D}^{-}(x) - \log \mathbb{D}^{-}(y) \mid \leq \frac{B}{d(x,\mathcal{C})^{\beta}} d(x,y)$.
		
	\end{itemize}
	
	With these conditions we can see that all the consequences for hyperbolic times are valid here and the expanding sets and measures are zooming sets and measures.
	
	\begin{definition}
		We say that a map is \emph{conformal at p} if $\mathbb{D}^{-}(p) = \mathbb{D}^{+}(p)$. So, we define
		\[
		\displaystyle \mathbb{D}(p) = \lim \frac{d(f(x),f(p)}{d(x,p)}.
		\]
	\end{definition}
	
	Now, we give an example of such an open non-uniformly expanding map.  
	
	\subsection{Expanding sets on a metric space} Let $\sigma : \Sigma_{2}^{+} \to \Sigma_{2}^{+}$ be the one-sided shift, with the usual metric:
	\[
	\displaystyle d(x,y) = \sum_{n=1}^{\infty} \frac{\mid x_{n} - y_{n} \mid}{2^{n}},
	\]
	where $x = \{x_{n}\}, y = \{y_{n}\}$. We have that $\sigma$ is a conformal map such that $\mathbb{D}^{-}(x) = 2, \text{for all} \, \, \, x \in \Sigma_{2}^{+}$. Also, every forward invariant set (in particular the whole $\Sigma_{2}^{+}$)  and all invariant measure
	for the shift $\sigma$ are expanding. 
	
	\subsection{Zooming sets on a metric space (not expanding)} Let $\sigma : \Sigma_{2}^{+} \to \Sigma_{2}^{+}$ be the one-sided shift, with the following metric for $\sum_{n=1}^{\infty} b_{n} < \infty$:
	\[
	\displaystyle d(x,y) = \sum_{n=1}^{\infty} b_{n}\mid x_{n} - y_{n} \mid,
	\]
	where $x = \{x_{n}\}, y = \{y_{n}\}$ and $b_{n+k} \leq b_{n}b_{k}$ for all $n,k \geq 1$. By induction, it means that $b_{n} \leq b_{1}^{n}$. Let us suppose that $b_{n} \leq a_{n}:=(n+b)^{-a}, a>1, b>0$ for all $n \geq 1$. 
	
	We claim that $a_{n}$ defines a Lipschitz contraction for the shift map. We require that there exists $n_{0} > 1$ such that $b_{n} > a_{1}^{n} \geq b_{1}^{n}$ for $n \leq n_{0}$. So, the contraction is not exponential. In fact, if $x,y$ belongs to the cylinder $C_{k}$ we have
	\begin{eqnarray*}
		\displaystyle d(x,y) &=& \sum_{n=1}^{\infty} b_{n}\mid x_{n} - y_{n} \mid = \sum_{n=k+1}^{\infty} b_{n}\mid x_{n} - y_{n} \mid = \sum_{n=1}^{\infty} b_{n+k}\mid x_{n+k} - y_{n+k} \mid\\
		&\leq& b_{k} \sum_{n=1}^{\infty} b_{n}\mid x_{n+k} - y_{n+k} \mid = b_{k} d(\sigma^{k}(x),\sigma^{k}(y)) \leq a_{k} d(\sigma^{k}(x),\sigma^{k}(y)).
	\end{eqnarray*}
	It implies that
	\begin{eqnarray*}
		\displaystyle d(\sigma^{i}(x),\sigma^{i}(y)) \leq  a_{k-i} d(\sigma^{k-i}(\sigma^{i}(x)),\sigma^{k-i}(\sigma^{i}(y)))= a_{k-i}d(\sigma^{k}(x),\sigma^{k}(y)), i \leq k.
	\end{eqnarray*}
	It means that the sequence $a_{n}$ defines a Lipschitz contraction, as we claimed.
	
	Every forward invariant set (in particular the whole $\Sigma_{2}^{+}$)  and all invariant measure
	for the shift $\sigma$ are zooming. 
	
	
	\paragraph{\bf Acknowledgement}{The authors would like to thank Artur O. Lopes for pointing out mistakes in the proof in a preliminary version of the text and for other valuable suggestions to improve this work. Also, the authors are thankful to the anonymous referees for contributing to improve the writing.}


\begin{thebibliography}{99}
		
		
		\bibitem{A}{\sc J. F. Alves}, {\it Statistical Analysis of Non-Uniformly Expanding Dynamical Systems}, $24^{\circ}$ Col\' oquio Brasileiro de Matem\' atica, 2003.
		
		\bibitem{A2} {\sc J. F. Alves}, {\it SRB Measures for Non-Hyperbolic Systems with Multidimensional Expansion}, Ann. Sci. \'Ecole Norm. Sup. 4, 33, 2000, 1-32.
		
		\bibitem{ABV}{\sc J. F. Alves, C. Bonatti, M. Viana}, {\it SRB Measures for Partially Hyperbolic Systems whose Central Direction is Mostly Expanding}, Invent. Math., 140, 2000, 351-398.
		
		
		
		\bibitem{AOS} {\sc J. F. Alves, K. Oliveira, E. Santana}, {\it Equilibrium States for Hyperbolic Potentials via Inducing Schemes}, Nonlinearity, 37, 2024, 095030.  
		
		
		\bibitem{AS} {\sc J. F. Alves, M. Soufi}, {\it Statistical Stability and Limit Laws for Rovella Maps}, Nonlinearity, 25, 2012, 3527-3552.
		
		\bibitem{AV} {\sc J. F. Alves, M. Viana}, {\it Statistical Stability for Robust Classes of Maps with Non-uniform Expansion}, Ergodic Theory and Dynamical Systems, 22, 2002, 1-32. 
		
		
		
		\bibitem{Bo1} {\sc T. Bousch}, {\it Le Poisson n'a pas d'Arêtes}, Ann. Inst. Henri Poincaré (Proba. et Stat.), 36, 2000, 489-508.

        \bibitem{Bo2} {\sc T. Bousch}, {\it La Condition de Walters}, Ann. Sci. ENS, 34, 2001, 287-311.

		
		\bibitem{BJ} {\sc T. Bousch, O. Jenkinson}, {\it Cohomology classes of Dynamically Non-negative $C^{k}$ Functions}, Inventiones Mathematicae, 148, 2002, 207-217.
		
		
		
		
		\bibitem{CLT} {\sc G. Contreras, A. O. Lopes, Ph. Thieullen}, {\it Lyapunov Minimizing Measures for Expanding Maps of the Circle}, Ergodic Theory and Dynamical Systems, 21, 05, 2001, 1379-1409.
				
		\bibitem{CG} {\sc J.-P. Conze, Y. Guivarc'h}, {\it Croissance des Somme Ergodiques et Principe Variationel}, Manuscript, 1993.
		
		
		
		
			
		\bibitem{F} {\sc J. M. Freitas}, {\it Continuity of SRB Measure and Entropy for Benedicks?Carleson Quadratic Maps}, Nonlinearity, 18, 2005, 831.

        \bibitem{HK}{\sc B. Hasselblat, A. Katok}, {\it Introduction to the Modern Theory of Dynamical Systems}, Encyclopedia of Mathematicas and its Applications, Cambridge University Press, 1995.
		
		
		
		
		
		
		
		
		
		
		\bibitem{OV}{\sc K. Oliveira, M. Viana}, {\it Foundations of Ergodic Theory}, Cambridge University Press, 2016.
		
		
		
		\bibitem{Pi1} {\sc V. Pinheiro}, {\it Expanding Measures}, Annales de l'Institute Henri Poincar\' e, 28, 2011, 889-939.
		
		
		\bibitem{PV} {\sc V. Pinheiro, P. Varandas}, {\it Thermodynamic Formalism for Expanding Measures}, arXiv:2202.05019
		
		\bibitem{R}{\sc R. Rockafellar}, {\it Convex Analysis}, Princeton University Press, 1970.
		
		
		
		
		
		
		
		\bibitem{Sa} {\sc E. Santana}, {\it Equilibrium States for Open Zooming Systems}, arXiv: 2010.08143 
		
		\bibitem{S} {\sc S. V. Savchenko}, {\it Cohomological Inequalities for Finite Topological Markov Chains} Functional Analysis and Its Applications, 33, 1999, 555-577. 
		
		
		\bibitem{VV} {\sc P. Varandas, M. Viana}, {\it Existence, Uniqueness and Stability of Equilibrium States for Non-Uniformly Expanding Maps}, Annales de l'Institute Henri Poincar\' e, 27, 2010, 555-593.
		
		\bibitem{V} {\sc M. Viana}, {\it Multidimensional NonHyperbolic Attractors},  Publications Math\' ematiques de l'Institut des Hautes \' Etudes Scientifiques, 85, 1, 1997, 63-96. 
		
		\bibitem{W} {\sc P. Walters}, {\it A Variational Principle for the Pressure of Continuous Transformations},
		American Journal of Mathematics, 97, 1975, 937-997. 
		
		
		
	\end{thebibliography}
\end{document}